\documentclass[11pt,reqno]{amsart}
\usepackage[T1]{fontenc}
\usepackage{graphicx, amssymb, amsmath}
\usepackage{cite}
\usepackage{color}
\definecolor{MyLinkColor}{rgb}{0,0,0.4}

\newcommand{\R}{{\mathbb R}}

\newcommand{\bX}{\mathbb{X}}
\newcommand{\bY}{\mathbb{Y}}
\newcommand{\bZ}{\mathbb{Z}}
\newcommand{\N}{{\mathbb N}}

\newcommand{\kH}{\mathcal{H}}
\newcommand{\cO}{\mathcal{O}}
\newcommand{\cF}{\mathcal{F}}
\newcommand{\cH}{\mathcal{H}}
\newcommand{\kL}{\mathcal{L}}

\newcommand{\wt}{\widetilde}

\newcommand{\ov}{\overline}
\newcommand{\p}{\partial}

\newcommand{\e}{\varepsilon}

\newcommand{\0}{\Omega}
\newcommand{\G}{\Gamma}

\newcommand{\tr}{\mathop{\rm tr}\nolimits}

\newcommand{\ima}{\operatorname{im}}
\newtheorem{thm}{Theorem}[section]
\newtheorem{prop}[thm]{Proposition}
\newtheorem{lemma}[thm]{Lemma}

\theoremstyle{remark} 
\newtheorem{rem}[thm]{Remark}

\numberwithin{equation}{section} 
\usepackage[colorlinks=true,linkcolor=MyLinkColor,citecolor=MyLinkColor]{hyperref}

\begin{document}

\title[Title]{Steady periodic hydroelastic waves in polar regions}
 
\thanks{}

\author{Bogdan-Vasile Matioc}
\address{Fakult\"at f\"ur Mathematik, Universit\"at Regensburg \\ D--93040 Regensburg, Deutschland}
\email{bogdan.matioc@ur.de}

\author{Emilian I. P\u ar\u au}
\address{School of Mathematics, University of East Anglia, Norwich NR4 7TJ, UK}
\email{e.parau@uea.ac.uk}

\begin{abstract}
We construct two-dimensional   steady periodic hydroelastic waves with vorticity that propagate on water of finite depth under a deformable floating elastic plate 
 which is modeled by using the special Cosserat theory of hyperelastic shells satisfying Kirchhoff's hypothesis.
This is achieved by providing necessary and sufficient condition for local bifurcation from the trivial branch of laminar flow solutions.
\end{abstract}

%%% NEED TO ADD MSC %%%

\maketitle

\pagestyle{myheadings}
\markboth{\sc{E. P\u ar\u au  \& B.-V.~Matioc}}{\sc{Steady periodic hydroelastic waves in polar regions}}

%%%%%%%%%%%%%%%%%%%%%%%%%%%%%%%%%%%%%%%%%%%%%%%%%%
%%%%%%%%%%%%%%%%%%%%%%%%%%%%%%%%%%%%%%%%%%%%%%%%%%
%%%%%%%%%%%%%%%%%%%%%%%%%%%%%%%%%%%%%%%%%%%%%%%%%%
%%%%%%%%%%%%%%%%%%%%%%%%%%%%%%%%%%%%%%%%%%%%%%%%%%
 \section{Introduction}\label{Sec:1}
%%%%%%%%%%%%%%%%%%%%%%%%%%%%%%%%%%%%%%%%%%%%%%%%%%
%%%%%%%%%%%%%%%%%%%%%%%%%%%%%%%%%%%%%%%%%%%%%%%%%%
%%%%%%%%%%%%%%%%%%%%%%%%%%%%%%%%%%%%%%%%%%%%%%%%%%
%%%%%%%%%%%%%%%%%%%%%%%%%%%%%%%%%%%%%%%%%%%%%%%%%%
Hydroelastic waves propagate in polar regions at the surface of  water   covered by a deformable ice sheet.
The water is modeled as an inviscid and incompressible fluid with constant  density (set to be $1 $). 
Assuming that   the  flow is two-dimensional, the equations of motion are the Euler equations
 \begin{subequations}\label{Euler}
\begin{equation}\label{E:1} 
\left.
\arraycolsep=1.4pt
\begin{array}{rllllll}
u_t+   u u_x+  vu_y&=&-P_x, \\[1ex]
v_t+  uv_x+ vv_y&=&-P_y-g, \\[1ex]
u_x+v_y&=&0,
\end{array}
\right\}\qquad -d<y<\eta(t,x),
\end{equation}
where  $u$ is the horizontal velocity, $v$ is the vertical velocity, $P$ is the pressure, and $g$ is the gravitational acceleration. 
The fluid domain is bounded from below by a flat impermeable bed located at $y=-d$, where $d>0$ is a constant,
 and   the  free wave surface $\{y=\eta(t,x)\}$  is assumed to be a thin ice sheet which is modeled as  a thin elastic plate by using the  Cosserat theory of hyperelastic shells satisfying Kirchhoff's hypothesis \cite{PT11, GP12}.
  The inertia of this thin elastic plate is neglected, we assume that the plate is not  pre-stressed, and consider only the effect of bending, neglecting the stretching of the plate.
Therefore  we  impose the following boundary conditions
\begin{equation} \label{E:2}
\left.
\arraycolsep=1.4pt
\begin{array}{rlllllll}
P&=&\alpha H(\eta)& \quad \text{on $y=\eta(t,x),$}\\[1ex]
v&=&\eta_t +u\eta_x&\quad \text{on $y=\eta(t,x),$}\\[1ex]
v&=&0&\quad\text{on $ y=-d,$}
\end{array}
\right\}
\end{equation}
where $\alpha>0$ is a constant, 
\[
H(\eta):=\omega(\eta)^{-1}\big[\omega(\eta)^{-1}\big(\omega(\eta)^{-3}\eta_{xx}\big)_{x}\big]_x+\frac{1}{2}\big(\omega(\eta)^{-3}\eta_{xx}\big)^3,
\]
and 
\begin{equation} \label{E:3}
\omega(\eta):=(1+\eta_x^2)^{1/2}.
\end{equation}  

 We investigate herein the existence of periodic steady water wave solutions to \eqref{E:1}-\eqref{E:3} for which  the unknowns $u,\, v,\, P,\, \eta$ satisfy 
 \[
 (u,v,P)(t,x,y)= (u,v,P)(x-ct,y) \qquad\text{and}\qquad \eta(t,x)= \eta(x-ct),
 \]
 where $c>0$ is the   speed of the wave. 
Moreover, letting $\lambda>0$ denote the period of the wave, we  restrict to solutions  which fulfill
\begin{equation} \label{E:4}
\int_0^\lambda \eta(x)\, {\rm d}x=0.
\end{equation}
 \end{subequations}
 
 Setting 
 \[
\0_\eta:=\{(x,y)\in\R^2\,:\, -d<y<\eta(x)\},
 \]
we thus look for functions $u,\, v,\, P:\overline{\0_\eta}\to\R$ and $\eta:\R\to\R$  which are $\lambda$-periodic in $x$  and  solve (after replacing $u-c$ by $u$) the coupled system of equations
\begin{subequations}\label{For1}
\begin{equation}\label{P1} 
\left.
\arraycolsep=1.4pt
\begin{array}{rllllll}
 u u_x+  vu_y&=&-P_x&\quad\text{in $\0_\eta$}, \\[1ex]
uv_x+ vv_y&=&-P_y-g&\quad\text{in $\0_\eta$}, \\[1ex]
u_x+v_y&=&0&\quad\text{in $\0_\eta$},\\[1ex]
P&=&\alpha H(\eta)&\quad\text{on $y=\eta(x),$}\\[1ex]
v&=&u\eta'&\quad\text{on $y=\eta(x),$}\\[1ex]
v&=&0&\quad\text{on $ y=-d,$}\\
\displaystyle\int_0^\lambda \eta(x)\, {\rm d}x&=&0.
\end{array}
\right\}
\end{equation}
We also exclude the presence of stagnation pointy by requiring that 
\begin{equation}\label{For1'}
u<0\quad\text{in $\ov{\0_\eta}$}.
\end{equation}
\end{subequations}

A similar setting has been considered in \cite{GHW16} where,  using a variational approach, the authors establish the existence of hydroelastic solitary waves
  for sufficiently large values of the dimensionless parameter $\alpha$  under the assumption that the vorticity is zero.
 Within the same irrotational scenario, the authors of \cite{AG23} establish the existence of  symmetric envelope hydroelastic solitary waves 
   by using spatial dynamics techniques.
  Moreover, in \cite{AAS16, AAS19},  the existence of periodic hydroelastic waves  which may posses a  multi-valued height between  two  superposed irrotational fluid layers with positive densities  separated by an elastic plate
  was shown via global bifurcation theorem, the analysis being based on the reformulation of  the problem as a vortex sheet problem.

We also mention the paper \cite{BT11} where, in the rotational setting,  weak periodic solutions to a related problem which acconts also for  surface tension effects  at the free boundary  
 are constructed via a variational approach, see also  \cite{T08, BaT11} for hydroelastic wave models which allow  for both bending and stretching of the elastic  wave surface.

The initial-value problem for flexural-gravity waves has been investigated in \cite{AS17} by developing
a well-posedness theory based on a vortex sheet formulation.

For numerical studies of hydroelastic waves, in the setting of constant vorticity, we refer to the recent works  \cite{GWM19, WV20, GMV19}.

In the present paper we extend the existence theory for~\eqref{For1} by allowing for a general vorticity~$$\overline\omega:=u_y-v_x.$$ 
The vorticity is a very important aspect of ocean flows also in polar regions, a non-zero vorticity characterizing  waves that interact  with non-uniform currents such as  the  Antarctic Circumpolar Current 
or near-surface currents in the Arctic Ocean,
 see e.g.~\cite{AC23, CJ23, Vallis_2017, MQ22}.
An essential tool in our  analysis is the availability of two equivalent formulations of \eqref{For1}, the stream function formulation \eqref{FOR2}    and the height function formulation~\eqref{FOR3}, 
see~Proposition~\ref{P:Ep}.
In particular,  the condition~\eqref{For1'} enables us to introduce the so-called vorticity function~${\gamma:[p_0,0]\to\R}$, where $p_0<0$ is a constant, which 
determines, via the stream function formulation~\eqref{FOR2}, the vorticity~$\overline\omega$ of the flow, see Section~\ref{Sec:2}.
While we consider a general H\"older continuous vorticity function $\gamma$, in order to establish the existence (and uniqueness) of a laminar flow solution to \eqref{For1} 
(with a flat  wave profile located at~${y=0}$ and $x$-independent velocity and pressure) 
the  restriction~\eqref{COND1Thm} is required on~$\gamma$ and the physical parameters. 
This laminar flow solution is a solution to \eqref{For1} for each value of the wavelength  $\lambda$.
We will  then use $\lambda$ as a bifurcation parameter in order to determine other nonlaminar symmetric (with respect to the horizontal line $x=0$) hydroelastic waves.
It turns out that bifurcation can occur if and only if   a second   condition, see \eqref{COND2Thm}, is satisfied.
In the setting of irrotational waves these conditions are explicit, see Remark~\ref{R:2}.
In order to prove our main result in Theorem~\ref{MT1}, we cannot directly use the aforementioned formulations \eqref{FOR2} or \eqref{FOR3} of the problem because
 the   boundary condition  in these formulations that corresponds to the dynamic boundary condition \eqref{P1}$_4$  involves fourth order derivatives of the unknown,
  whereas the elliptic equation posed in the (fluid) domain is of second order.
 However, inspired by an idea used also in other  steady water wave problems, see \cite{HM14b, M14, MM14, ASW16, CM14b, WW23, EKLM20}, we may reformulate~\eqref{FOR3}, after rescaling the horizontal variable by $\lambda$,
 as a quasilinear elliptic equation subject to a  boundary condition which may be viewed as a compact, but at the same time  nonlocal and nonlinear, perturbation of the trace operator, see~\eqref{PBF}.
 The wavelength~$\lambda$ appears as a free parameter in~\eqref{PBF} and we show that the local bifurcation theorem of Crandall and Rabinowitz, cf.  \cite[Theorem~1.7]{CR71}, can be applied in the context 
 of~\eqref{PBF} to prove our main result and  establish in this way the existence  of solutions to \eqref{For1} within the regularity class introduced in Proposition~\ref{P:Ep}.

\begin{thm}\label{MT1} Let $\alpha>0$, $d>0$, $p_0<0,$  and $g>0$ be fixed  and choose $\beta\in(0,1)$. Assume that the vorticity  function $\gamma$ belongs to   $ {\rm C}^\beta([p_0,0])$ and set
\[
\Gamma(p):=\int_{p_0}^p\gamma(s)\,{\rm d}s,\qquad p\in[p_0,0].
\]
Then we have:
\begin{itemize}
\item[(i)] The problem \eqref{For1}  has laminar solutions $(u,\, v,\, P,\,\eta)=(u_*,v_*,P_*,0)$, with~${u_*,\,v_*,\,P_*}$  independent of the $x$-variable, iff
\begin{equation}\label{COND1Thm}
\lim_{\vartheta\searrow2\underset{[p_0,0]}\max\G} \int_{p_0}^0(\vartheta-2\G(s))^{-1/2}{\rm d}s>d.
\end{equation}
If~\eqref{COND1Thm} is satisfied, there exists exactly one laminar solution $(u_*,v_*,P_*,0)$ to~\eqref{For1}.
\item[(ii)] Assume that the condition \eqref{COND1Thm} holds true and set $a:=(\vartheta-2\G)^{1/2}\in{\rm C}^{1+\beta}([p_0,0]),$ where $\vartheta>2\,\underset{[p_0,0]}\max\,\G $ is the unique constant which  satisfies
\begin{equation*} 
\int_{p_0}^0(\vartheta-2\G(s))^{-1/2}{\rm d}s=d.
\end{equation*}
 Then:
 \begin{itemize}
\item[(iia)] If  
\begin{equation}\label{COND2Thm}
g  \int_{p_0}^0\frac{1}{a^{3}(p)}{\rm d}p< 1
\end{equation}
does not hold, there exist no solutions  to \eqref{For1} which bifurcate from the trivial branch of laminar flow solutions $\{(\lambda, u_*,v_*,P_*,0)\,:\, \lambda>0\};$ 
\item[(iib)] If \eqref{COND2Thm} is satisfied,  there exists a unique minimal wavelength~$\lambda_*>0$ with the property that~$(\lambda_*, u_*,v_*,P_*,0)$ is a local bifurcation point (of the trivial branch) of solutions to \eqref{For1}.
More precisely, there exists a local bifurcation curve  
 \[
\mathcal{C}=\{(\lambda(s), u(s),v(s),P(s),\eta(s))\,:\, s\in(-\e,\e)\},
 \]
  where $\e>0$ is a small constant, having  the following properties:
    \begin{itemize}
    \item[$\bullet$] $[s\mapsto \lambda(s)]$ is smooth, $\lambda(s)>0$ for $s>0,$ and $\lambda(s)=\lambda_*+O(s)$ for $s\to0$;
    \item[$\bullet$]  $(u(0),v(0),P(0),\eta(0))=(u_*,v_*,P_*,0)$;
    \item[$\bullet$] For $s\neq0,$ the tupel $(u(s),v(s),P(s),\eta(s))$ is a  solution  to \eqref{For1} with minimal wavelength $\lambda(s)$  and vorticity function $\gamma$. 
    Moreover, the wave profile  has  one crest (located on the vertical line $x=0$) and one trough per period, is symmetric with respect to crest and trough lines, 
    and  strictly monotone between crest and trough.
\end{itemize}
\end{itemize}
\end{itemize}
\end{thm}

Concerning  Theorem~\ref{MT1}, we add the following remarks.
\begin{rem}\label{R:1}\phantom{A}
\begin{itemize}
\item[(i)] If $0\neq |s|<\e$, then $(u(s),v(s),P(s),\eta(s))$ is also of solution to \eqref{For1} having (not minimal) period $k\lambda(s)$  for all $1\leq k\in\N$.
In particular, for each $k\geq 1$,   $(k\lambda_*, u_*,v_*,P_*,0)$ is also a local bifurcation point of the trivial branch of solutions to \eqref{For1}.
In Theorem~\ref{MT1} we prove that these are the only points on the trivial branch of solutions from where other symmetric solutions bifurcate. 
\item[(ii)] Our analysis discloses, under the assumption~\eqref{COND2Thm}, that  bifurcation from double (actually multiple) eigenvalues of symmetric waves is excluded along the trivial branch of laminar solutions to the hydroelastic waves problem \eqref{For1}.
\end{itemize}
\end{rem}

We now illustrate the conditions for bifurcation from Theorem~\ref{MT1} in the particular case of irrotational waves (with $\gamma=0$). 
\begin{rem}\label{R:2} If $\gamma=0$, then \eqref{COND1Thm} is automatically fulfilled and the inequality \eqref{COND2Thm} is equivalent to
\[
\frac{gd^3}{p_0^2}<1.
\]
Moreover, the  wavelength  $\lambda_*>0$ can be determined as the unique solution to the equation
\begin{equation}\label{disperrel}
\Big[g+\alpha\Big(\frac{2\pi}{\lambda_*}\Big)^4\Big]\tanh\Big(\frac{2\pi d}{\lambda_*}\Big)=\frac{p_0^2}{d^2}\frac{2\pi}{\lambda_*}.
\end{equation}
Equation~\eqref{disperrel} is the dispersion relation for irrotational hydroelastic waves.
\end{rem}

\noindent{Outline:} In Section~\ref{Sec:2} we present two further equivalent formulations of \eqref{For1}: the stream function formulation~\eqref{FOR2} and the height function formulation~\eqref{FOR3}. 
Then, in Section~\ref{Sec:3}, we reformulate~\eqref{FOR3} by reexpressing the boundary condition in~\eqref{FOR3} obtained from the dynamic boundary condition as a compact, but nonlinear and nonlocal, perturbation of a Dirichlet boundary condition, see~\eqref{PBF}.
Finally, in Section~\ref{Sec:4}, we recast~\eqref{PBF} as a bifurcation problem and prove Theorem~\ref{MT1}.

%%%%%%%%%%%%%%%%%%%%%%%%%%%%%%%%%%%%%%%%%%%%%%%%%%
%%%%%%%%%%%%%%%%%%%%%%%%%%%%%%%%%%%%%%%%%%%%%%%%%%
%%%%%%%%%%%%%%%%%%%%%%%%%%%%%%%%%%%%%%%%%%%%%%%%%%
%%%%%%%%%%%%%%%%%%%%%%%%%%%%%%%%%%%%%%%%%%%%%%%%%%
\section{Equivalent formulations of the hydroelastic waves problem}\label{Sec:2}
%%%%%%%%%%%%%%%%%%%%%%%%%%%%%%%%%%%%%%%%%%%%%%%%%%
%%%%%%%%%%%%%%%%%%%%%%%%%%%%%%%%%%%%%%%%%%%%%%%%%%
%%%%%%%%%%%%%%%%%%%%%%%%%%%%%%%%%%%%%%%%%%%%%%%%%%
%%%%%%%%%%%%%%%%%%%%%%%%%%%%%%%%%%%%%%%%%%%%%%%%%%
In this section we introduce two further equivalent formulations of the hydroelastic waves problem \eqref{For1} which have  been useful also when constructing rotational water waves in other physical scenarios,
 cf. e.g. \cite{Con11, HM15, SUR22}.  

\subsection{The velocity formulation} The stream function $\psi:\overline{\0_\eta}\to\R$  is defined by the equations
\[
\psi=0 \quad\text{on $y=\eta(x)$}\qquad\text{and}\qquad \nabla\psi=(-v,u)\quad\text{in $\0_\eta$}.
\]
Since $\psi_y<0$, cf. \eqref{For1'}, the constant $p_0:=-\psi|_{y=-d}$ is negative.

Let further $\mathcal{H}:\ov{\0_\eta}\to\ov\0$, where $\0:=\R\times(p_0,0),$ be defined by the formula
\[
\kH(x,y):=(q(x,y),p(x,y)):=(x,-\psi(x,y)).
\]
As a consequence of  \eqref{For1'}, the function $\kH$ is a bijection.
For smooth solutions to \eqref{For1} we then compute
\begin{align*}
\p_q(\overline\omega\circ \kH^{-1}) =\Big(\overline\omega_x+\frac{v}{u}\overline\omega_y\Big)\circ \kH^{-1}=0\quad\text{in $\0$,}
\end{align*}
since \eqref{P1}$_1$-\eqref{P1}$_3$ yield 
\[
u\overline\omega_x+v\overline\omega_y=0\quad\text{in $\0_\eta$}.
\]
Hence,  there exists a function $\gamma:[p_0,0]\to\R$, the so-called vorticity function, with the property  that $\overline\omega\circ \kH^{-1}(q,p)=\gamma(p)$ for all $(q,p)\in\ov\0,$ or equivalently 
\[
\overline\omega(x,y)=\gamma(-\psi(x,y))\quad\text{for all $(x,y)\in\ov{\0_\eta}$.}
\]
This relation together with \eqref{P1}$_1$-\eqref{P1}$_2$ implies that the energy
\[
E:=P+\frac{u^2+v^2}{2}+gy-\int_0^\psi\gamma(-s) \, {\rm d}s 
\]
is constant in $\0_\eta.$
Evaluating  this expression at the wave surface, we deduce together with  the relation~\eqref{P1}$_4$, that 
\begin{equation}\label{Ber1}
|\nabla \psi|^2+2g\eta+2\alpha H(\eta)=Q\qquad\text{on $y=\eta(x)$,}
\end{equation}
where $Q$ is a constant.
Integration  by parts further leads to  
\[
\int_0^\lambda  H(\eta)\, {\rm d}x=0,
\]
and, since also $\eta$ has zero  integral mean, we infer from \eqref{Ber1}, after integrating over one period, that 
\[
Q=\frac{1}{\lambda}\int_0^\lambda|\nabla \psi|^2(x,\eta(x))\, {\rm d}x.
\]
Consequently, $\psi$ solves the boundary value problem 
\begin{subequations}\label{FOR2}
\begin{equation}\label{For2}
\left.
\arraycolsep=1.4pt
\begin{array}{rllllll}
\Delta\psi&=&\gamma(-\psi)&\quad\text{in $\0_\eta$},\\[1ex]
\psi&=&0&\quad\text{on $y=\eta(x),$}\\[1ex]
\psi&=&-p_0&\quad\text{on $y=-d,$}\\[1ex]
|\nabla \psi|^2+2g\eta+2\alpha H(\eta)&=&\displaystyle\frac{1}{\lambda}\int_0^\lambda|\nabla \psi|^2(x,\eta(x)\, {\rm d}x&\quad\text{on $y=\eta(x)$}
\end{array}
\right\}
\end{equation}
and satisfies
\begin{equation}\label{For2'}
\psi_y<0\quad\text{in $\ov{\0_\eta}$}.
\end{equation}
\end{subequations}

\subsection{The height function formulation} 

We define the height function $h:\ov\0\to\R$ by 
\[
h(q,p)=y,
\]
which associates to a point $(q,p)\in\ov\0$  the  vertical coordinate  of the fluid particle located at~${(x,y)=\cH^{-1}(q,p)}\in\ov{\0_\eta}$.
Then, since  $\eta=h(\cdot,0)$, the function $h$ solves the following boundary value problem
\begin{subequations}\label{FOR3}
\begin{equation}\label{For3}
\left.
\arraycolsep=1.4pt
\begin{array}{rllllll}
(1+h_q^2)h_{pp}-2h_ph_qh_{pq}+h_p^2h_{qq}-\gamma h_p^3&=&0&\quad\text{in $\0$,}\\[1ex]
h&=&-d&\quad\text{on $p=p_0$,}\\[1ex]
\displaystyle\frac{1+h_q^2}{h_p^2}+2gh+2\alpha H(h)&=& \displaystyle\frac{1}{\lambda}\int_0^\lambda\frac{1+h_q^2}{h_p^2}(q,0)\, {\rm d}q&\quad\text{on $p=0$,}
\end{array}
\right\}
\end{equation}
together with 
\begin{equation}\label{For3'}
h_p>0\quad\text{in $\ov{\0}$}.
\end{equation}
\end{subequations}

\begin{prop}[Equivalence of formulations]\label{P:Ep} 
Let  $\beta\in(0,1).$ 
 Then, the following formulations are equivalent:
\begin{itemize}
\item[(i)] The velocity formulation \eqref{For1} for \[\text{ $u,\,v,\,P\in {\rm C}^{1+\beta}(\ov{\0_\eta}) $ and 
$\eta\in {\rm C}^{4}(\R)$;}\]
\item[(ii)] The stream function formulation  \eqref{FOR2} for
\[\text{ $\psi\in {\rm C}^{2+\beta}(\ov{\0_\eta}) $, $\eta\in {\rm C}^{4}(\R) $, and~${\gamma\in{\rm C}^{\beta}([p_0,0])}$;}\]
\item[(iii)] The height function formulation \eqref{FOR3} for 
\[\text{$h \in{\rm C}^{2+\beta}(\ov\0) $ with $\tr_0 h\in {\rm C}^{4}(\R)$,   and $\gamma\in{\rm C}^{\beta}([p_0,0])$.}\]
\end{itemize}
\end{prop}
\begin{proof}
The proof is similar to that of \cite[Lemma 2.1]{CoSt04}.
\end{proof}

 %%%%%%%%%%%%%%%%%%%%%%%%%%%%%%%%%%%%%%%%%%%%%%%%%%
%%%%%%%%%%%%%%%%%%%%%%%%%%%%%%%%%%%%%%%%%%%%%%%%%%
%%%%%%%%%%%%%%%%%%%%%%%%%%%%%%%%%%%%%%%%%%%%%%%%%%
%%%%%%%%%%%%%%%%%%%%%%%%%%%%%%%%%%%%%%%%%%%%%%%%%%
 \section{An equivalent formulation of   \eqref{FOR3}}\label{Sec:3}
 %%%%%%%%%%%%%%%%%%%%%%%%%%%%%%%%%%%%%%%%%%%%%%%%%%
%%%%%%%%%%%%%%%%%%%%%%%%%%%%%%%%%%%%%%%%%%%%%%%%%%
%%%%%%%%%%%%%%%%%%%%%%%%%%%%%%%%%%%%%%%%%%%%%%%%%%
%%%%%%%%%%%%%%%%%%%%%%%%%%%%%%%%%%%%%%%%%%%%%%%%%%
 In our analysis we will take advantage of  the height function formulation~\eqref{FOR3} to establish the existence of steady periodic hydroelastic waves.
 The main tool used to achieve this goal is the local bifurcation theorem of Crandall and Rabinowitz, cf.  \cite[Theorem~1.7]{CR71}.
 The appropriate parameter for bifurcation is the wavelength $\lambda>0$.
 Since $h$ is $\lambda$-periodic with respect to $q$ it is therefore suitable to rescale $h$  according to 
\begin{equation}\label{Scal}
\wt h(q,p):= h(\lambda q,p),\qquad (q,p)\in\ov\0.
\end{equation}
The function $\wt h$  is $1$-periodic and solves (after dropping  tildes) the equations
\begin{subequations}\label{FOR4}
\begin{equation}\label{For4}
\left.
\arraycolsep=1.4pt
\begin{array}{rllllll}
(\lambda^2+h_q^2)h_{pp}-2h_ph_qh_{pq}+h_p^2h_{qq}-\lambda^2\gamma h_p^3&=&0&\quad\text{in $\0$,}\\[1ex]
h&=&-d&\quad\text{on $p=p_0$,}\\[1ex]
\displaystyle\frac{\lambda^2+h_q^2}{h_p^2}+2g\lambda^2h+\frac{2\alpha}{\lambda}H \Big(\frac{h}{\lambda}\Big)&=& \displaystyle\int_0^1\tr_0\frac{\lambda^2+h_q^2}{h_p^2} \, {\rm d}q&\quad\text{on $p=0 $}
\end{array}
\right\}
\end{equation}
and 
\begin{equation}\label{For4'}
h_p>0\quad\text{in $\ov{\0}$}.
\end{equation}
\end{subequations}
In the following $\tr_0$ is the trace operator with respect to the boundary component $\{p=0\}$ of $\0$, that is, given $f:\ov\Omega\to\R$, 
   the function $\tr_0 f:\R\to\R$ is defined by $\tr_0 f(q)=f(q,0)$ for~$q\in\R$.

Let $\beta\in(0,1) $ be fixed.
We will assume that $h\in\bX$, where the Banach spaces $\bX$ is defined as follows
\[
\bX:=\Big\{h\in{\rm C}^{2+\beta}(\ov\0)\,:\, \text{$h$ is even and $1$-periodic with repect to $q$ and $\int_0^1\tr_0h\, {\rm d}q=0$}\Big\}.
\]
Similarly, given $k\in\N$, the space ${\rm C}_e^{k+\beta}(\R)$  consists of the even and $1$-periodic  functions with uniformly $\beta$-H\"older continuous $k$th derivative.
Moreover, ${\rm C}_{e,0}^{k+\beta}(\R)$  is the subspace of~${{\rm C}_e^{k+\beta}(\R)}$ which contains only functions with zero integral mean.
We note that the boundary condition~\eqref{For4}$_3$ is not well-defined for~$h\in\bX$ as fourth order derivatives of $h$ appear in  this equation. 
However, since of $H(h/\lambda)$ involves only derivatives of~$h$ with respect to the horizontal  variable $q,$ we may reformulate the boundary condition \eqref{For4}$_3$  as a nonlocal and nonlinear  compact  
 perturbation of the trace operator $\tr_0$.
 
 To this end we set  $\zeta:=\tr_0 h/\lambda $ and note that, if $h\in\bX$ satisfies $\tr_0h\in{\rm C}_e^{4+\beta}(\R) $, \eqref{For4'}, and the  boundary condition  \eqref{For4}$_3$, then
 \[
H(\zeta)=B(\lambda,h), 
 \]
where $B:(0,\infty)\times\{h\in\bX\,:\,  \text{$h_p>0$ in $\ov\0$}\}\to  {\rm C}_{e,0}^{1+\beta}(\R)$ is the smooth mapping defined by
 \begin{equation}\label{Op:B}
B(\lambda,h):=\frac{\lambda}{2\alpha} \bigg[\int_0^1\tr_0\frac{\lambda^2+h_q^2}{h_p^2}\, {\rm d}q-\tr_0\Big(\frac{\lambda^2+h_q^2}{h_p^2} +2g\lambda^2h\Big)\bigg].
 \end{equation}
Integration leads to 
 \[
\int_0^x H(\zeta) \, {\rm d}s=\int_0^x B(\lambda,h) \, {\rm d}s\qquad\text{for all $x\in\R$.}
 \]
 In view of the fact that  $\tr_0h$ is an even function we obtain that 
  \[
\int_0^x  H(\zeta) \, {\rm d}s=\big[\omega^{-2}(\zeta)(\omega^{-3}(\zeta)\zeta'')'
+\frac{1}{2}\zeta'\zeta''^2\omega^{-7}(\zeta)\big](x),\qquad x\in\R,
 \]
where $\omega(\cdot)$ is the nonlinear operator defined in \eqref{E:3}, hence
 \begin{equation}\label{eq:zeta'''}
\Big[\omega^{-2}(\zeta)(\omega^{-3}(\zeta)\zeta'')'+\frac{1}{2}\zeta'\zeta''^2\omega^{-7}(\zeta)\Big](x)=\int_0^x B(\lambda,h) \, {\rm d}s\qquad\text{  for all $x\in\R$.}
 \end{equation}
 Integrating the last relation once more we arrive at
 \begin{equation}\label{eq:zeta''}
\zeta''(q)= \omega^{5} (\zeta)(q)\Big(C+\int_0^q\int_0^x B(\lambda,h) \, {\rm d}s\,{\rm d}x-\frac{5}{2}\int_0^q \zeta' \zeta''^2\omega^{-7}(\zeta)\, {\rm d}x\Big)
 \end{equation}
 for all $q\in\R$, where $C:=\zeta''(0).$
 Letting $\Phi:(0,\infty)\times \{h\in\bX\,:\, \text{$h_p>0$ in $\ov\0$}\}\to {\rm C}^{1+\alpha}_e(\R)$ be defined by
  \begin{equation}\label{Op:Phi}
\Phi(\lambda,h)(q):= \int_0^q\int_0^x B(\lambda,h) \, {\rm d}s\,{\rm d}x
-\frac{5}{2}\int_0^q \Big(\frac{\tr_0 h}{\lambda}\Big)' \Big[\Big(\frac{\tr_0 h}{\lambda}\Big)''\Big]^2\omega^{-7}\Big(\frac{\tr_0 h}{\lambda}\Big) \, {\rm d}x
\end{equation}
 for $q\in\R,$ the previous equality identifies, since $\zeta''$ has zero integral mean, the constant $C$ as
 \[
C=-\Big(\int_0^1\omega^{5} (\zeta){\rm d}q\Big)^{-1}\int_0^1 \omega^5(\zeta) \Phi(\lambda,h)\,{\rm d}q, 
 \] 
 and therefore we have
 \[
(1-\p_q^2)\zeta=\zeta+\omega^5(\zeta) \Big(\int_0^1\omega^{5} (\zeta){\rm d}q\Big)^{-1}\int_0^1 \omega^5(\zeta) \Phi(\lambda,h){\rm d}q-\omega^5(\zeta) \Phi(\lambda,h)\in {\rm C}^{1+\alpha}_{e,0}(\R).
 \]
Since $1-\p_q^2:{\rm C}^{3+\alpha}_{e,0}(\R)\to{\rm C}^{1+\alpha}_{e,0}(\R)$ is an isomorphism, we get 
 \[
\zeta=(1-\p_q^2)^{-1}\Big[\zeta+\omega^5(\zeta) \Big(\int_0^1\omega^{5} (\zeta)\,{\rm d}q\Big)^{-1}\int_0^1 \omega^5(\zeta) \Phi(\lambda,h)\,{\rm d}q-\omega^5(\zeta) \Phi(\lambda,h)\Big]\in {\rm C}^{3+\alpha}_{e,0}(\R).
 \]
 This proves that $\tr_0h$ satisfies \eqref{eq:equiv}  and therewith the first implication  in Lemma~\ref{L:1} below. 

 \begin{lemma}\label{L:1} Let $h\in\bX$ satisfy \eqref{For4'}. Then the following are equivalent:
 \begin{itemize}
\item[(i)] $\tr_0h\in{\rm C}_e^{4+\beta}(\R) $ and $h$ satisfies \eqref{For4}$_3$;
\item[(ii)] With  $\Phi$ defined in \eqref{Op:Phi} we have
\begin{equation}\label{eq:equiv}
\begin{aligned}
\tr_0h&=(1-\p_q^2)^{-1}\Big[\lambda\omega^5(\tr_0h/\lambda) \Big(\int_0^1\omega^{5} (\tr_0h/\lambda)\,{\rm d}q\Big)^{-1}\int_0^1 \omega^5(\tr_0h/\lambda) \Phi(\lambda,h)\,{\rm d}q\\[1ex]
&\hspace{2.5cm}+\tr_0h-\lambda\omega^5(\tr_0h/\lambda) \Phi(\lambda,h)\Big].
\end{aligned}
\end{equation}
 \end{itemize}
 \end{lemma}
 \begin{proof} It remains to prove that (ii) implies (i).
 Let thus Lemma~\ref{L:1}~(ii) be satisfied. 
 Then, since the argument of $(1-\p_q^2)^{-1}$ in \eqref{eq:equiv} lies in~${{\rm C}_{e,0}^{1+\beta}(\R)}$, the function~${\zeta:=\tr_0 h/\lambda}$ belongs to ${\rm C}^{3+\alpha}_{e,0}(\R)$ and
 satisfies \eqref{eq:zeta''}.  
 Multiplying now~\eqref{eq:zeta''} by $\omega^{-5}(\zeta)$ and differentiating the resulting equation once, we deduce that $\zeta$ satisfies the equation \eqref{eq:zeta'''}, hence $\zeta\in{{\rm C}_e^{4+\beta}(\R)}.$
 Differentiating  \eqref{eq:zeta'''}, we deduce that indeed~$H(\zeta)=B(\lambda,h)$, thus \eqref{For4}$_3$ holds true.
 \end{proof}

In view of Lemma~\ref{L:1} we have formulated the problem \eqref{FOR4} as the following system
\begin{subequations}\label{PBF}
\begin{equation}\label{PBF1}
\left.
\arraycolsep=1.4pt
\begin{array}{rllllll}
(\lambda^2+h_q^2)h_{pp}-2h_ph_qh_{pq}+h_p^2h_{qq}-\lambda^2\gamma h_p^3&=&0&\quad\text{in $\0$,}\\[1ex]
h&=&-d\displaystyle&\quad\text{on $p=p_0$,}\\[1ex]
h&=&\Psi(\lambda,h)&\quad\text{on $p=0,$}
\end{array}
\right\}
\end{equation}
and 
\begin{equation}\label{PBF2}
 h_p>0\quad\text{in $\ov{\0}$},
\end{equation}
\end{subequations}
where $\Psi:(0,\infty)\times\{h\in\bX\,:\,  \text{$h_p>0$ in $\ov\0$}\}\to {\rm C}^{3+\alpha}_{e,0}(\R)$ is the smooth mapping given by
\begin{equation}\label{eq:psi}
\begin{aligned}
\Psi(\lambda,h)&:=(1-\p_q^2)^{-1}\Big[\lambda\omega^5(\tr_0 h/\lambda) \Big(\int_0^1\omega^{5} (\tr_0h/\lambda)\,{\rm d}q\Big)^{-1}\int_0^1 \omega^5(\tr_0h/\lambda) \Phi(\lambda,h)\,{\rm d}q\\[1ex]
&\hspace{2.5cm}+\tr_0h-\lambda\omega^5(\tr_0h/\lambda) \Phi(\lambda,h)\Big].
\end{aligned}
\end{equation}

%%%%%%%%%%%%%%%%%%%%%%%%%%%%%%%%%%%%%%%%%%%%%%%%%%
%%%%%%%%%%%%%%%%%%%%%%%%%%%%%%%%%%%%%%%%%%%%%%%%%%
%%%%%%%%%%%%%%%%%%%%%%%%%%%%%%%%%%%%%%%%%%%%%%%%%%
%%%%%%%%%%%%%%%%%%%%%%%%%%%%%%%%%%%%%%%%%%%%%%%%%%
\section{Local bifurcation analysis}\label{Sec:4}
%%%%%%%%%%%%%%%%%%%%%%%%%%%%%%%%%%%%%%%%%%%%%%%%%%
%%%%%%%%%%%%%%%%%%%%%%%%%%%%%%%%%%%%%%%%%%%%%%%%%%
%%%%%%%%%%%%%%%%%%%%%%%%%%%%%%%%%%%%%%%%%%%%%%%%%%
%%%%%%%%%%%%%%%%%%%%%%%%%%%%%%%%%%%%%%%%%%%%%%%%%%
In this section we consider the equivalent formulation~\eqref{PBF} of the  hydroelastic waves problem \eqref{For1} and study its solutions set.
In  a first step we investigate in Section~\ref{Sec:41} the existence of laminar flow solutions to \eqref{PBF}.
Then, in Section~\ref{Sec:42} we formulate \eqref{PBF} as a bifurcation problem, see \eqref{BifE}, and determine a sufficient and necessary condition, see \eqref{COND2},  for   bifurcation from the set of laminar flow solutions.
We conclude this section with the proof of the main result.

\subsection{Laminar flow solutions for \eqref{PBF}}\label{Sec:41}
We next investigate the existence of laminar flow solutions to \eqref{PBF}, that is, given $\lambda>0$,  we look for solutions  $H=H(\lambda)\in\bX$ to~\eqref{PBF} that depend only on the variable $p$.
Then, $H\in{\rm C}^{2+\beta}([p_0,0])$     solves the Sturm–Liouville problem 
\begin{equation}\label{LFS}
\left.
\arraycolsep=1.4pt
\begin{array}{llllll}
H'' =\gamma H'^3\qquad\text{in $(p_0,0)$,}\\[1ex]
H(p_0)=-d,\quad H(0)=0
\end{array}
\right\}
\end{equation}
together with the inequality that $H'>0$ in $[p_0,0]$. 
 The next result shows that \eqref{COND1Thm} is a sufficient and necessary condition for the existence of a (unique) laminar flow solution.
 
 \begin{lemma}\label{L:0}
 The boundary value problem~\eqref{LFS} has a solution $H\in{\rm C}^{2+\beta}([p_0,0])$ with $H'>0$ in $[p_0,0]$ iff \eqref{COND1Thm} is satisfied.
 In this case the solution is unique and it is given by
 \begin{equation}\label{eq:H}
H(p)=-\int_{p}^0 (\vartheta-2\G(s))^{-1/2}{\rm d}s,\qquad p\in[p_0,0],
 \end{equation}
 where  $\vartheta>2\max_{[p_0,0]}\G$  is the unique solution to
 \begin{equation}\label{COND1}
 \int_{p_0}^0(\vartheta-2\G(s))^{-1/2}{\rm d}s=d.
 \end{equation}
 \end{lemma}
\begin{proof}
Since \eqref{LFS}$_1$ is equivalent to
\[
\Big(\frac{1}{H'^2}+2\Gamma\Big)'=0,
\]
we obtain that 
\[
\frac{1}{H'^2(p)}=\vartheta-2\Gamma(p),\qquad p\in[p_0,0],
\]
where the constant $\vartheta$ needs to satisfy $\vartheta>2\max_{[p_0,0]}\G$.  
From the latter relation we infer that~$H$ is given by \eqref{eq:H} and solves \eqref{LFS} iff $\vartheta$ is the solution to \eqref{COND1}.
In view of the monotonicity of the integrand in \eqref{COND1} with respect to $\vartheta$,
the existence of the (unique) solution to \eqref{COND1}  is equivalent to \eqref{COND1Thm}.
\end{proof}

\subsection{Bifurcation analysis for \eqref{PBF}}\label{Sec:42}
In the following we assume that \eqref{COND1Thm} is satisfied and we denote by $H$ the laminar flow solution identified in Lemma~\ref{L:0}.
We next define the Banach spaces $\bY$ and $\bZ_1\times\bZ_2$ consisting of $1$-periodic functions with respect to the variable~$q$ by setting
\begin{align*}
\bY:=\{h\in\bX\, :\,\text{$ h=0$ on ${p=p_0}$}\},\qquad\bZ_1:=\{h\in{\rm C}^\beta (\ov \0)\,:\, \text{$h$ is even}\},\qquad \bZ_2:={\rm C}_{e,0}^{2+\beta}(\R),
\end{align*}
and we denote by $\cO$ the open subset of $\bY$ defined by
\[
\cO:=\{h\in\bY\,:\, h_p+H'>0 \quad\text{in $\ov{\0}$}\}.
\]
We further  introduce the operator $\cF:=(\cF_1,\cF_2):(0,\infty)\times \cO\subset \R\times \bY\to\bZ_1\times\bZ_2$ by  
\begin{equation}
\begin{aligned}
\cF_1(\lambda,h)&:=(\lambda^2+h_q^2)(H''+h_{pp})-2(H'+h_p)h_qh_{pq}+(H'^2+h_p^2)h_{qq}-\lambda^2\gamma (H'+h_p)^3,\\[1ex]
\cF_2(\lambda,h)&:=\tr_0h-\psi(\lambda,h+H).
\end{aligned}
\end{equation} 
Hence, the problem \eqref{PBF} is equivalent to the nonlinear and nonlocal equation
\begin{equation}\label{BifE}
\cF(\lambda,h)=0,
\end{equation}
where 
\begin{equation}\label{refF}
\cF\in{\rm C}^\infty((0,\infty)\times\cO,\bZ_1\times\bZ_2)
\end{equation}
has the property that 
\begin{equation}\label{lamflsol}
\cF(\lambda,0)=0\qquad\text{for all $\lambda>0$.}
\end{equation}

Our goal is to apply the Crandall--Rabinowitz theorem  \cite[Theorem~1.7]{CR71} on bifurcation from simple eigenvalues in the context of \eqref{BifE}  in order to determine 
 new solutions to \eqref{BifE} which are also $q$-dependent.
For this reason we  shall determine  $\lambda_*>0$ with the property that  the partial Fr\'echet derivative $\partial_h\cF(\lambda_*,0)$ is a Fredholm operator of index zero with a one-dimensional kernel.

 Given $\lambda>0$, the   partial Fr\'echet derivative $\partial_h\cF(\lambda ,0):=(L,T)$ is given by
 \begin{equation}\label{derf}
 \begin{aligned}
 L[h]&:=\lambda^2h_{pp}+H'^2h_{qq}-3\lambda^2\gamma H'^2h_p,\\
 T[h]&:=\tr_0 h-(1-\p_q^2)^{-1}\Big[\tr_0h-\frac{\lambda^4}{\alpha}\Big(S[h]-\int_0^1S[h]\, {\rm d} q\Big)\Big]
 \end{aligned}
 \end{equation}
for $h\in\bY$, where
\[
S[h](q):=\int_0^q\int_0^x\Big[\tr_0\Big(\frac{h_p}{H'^3}- gh\Big) -\int_0^1\tr_0\frac{h_p}{H'^3}\,{\rm d}q\Big]\, {\rm d}s\,{\rm d}x,\qquad q\in\R.
\]
\begin{lemma}\label{L:2}
Given $\lambda>0$, the Fr\'echet derivative $\p_h\cF(\lambda,0)\in \kL(\bY, \bZ_1\times\bZ_2)$ is a Fredholm operator of index zero.
\end{lemma}
\begin{proof}
In view of \cite[Theorem 6.14]{GT01}, the operator $$(\lambda^2\p_p^2+H'^2\p_q^2,\tr):{\rm C}^{2+\beta}(\ov \0)\to {\rm C}^{\beta}(\ov \0)\times {\rm C}^{2+\beta}(\R)^2$$ is an isomorphism.
We may infer from this property that $(\lambda^2\p_p^2+H'^2\p_q^2,\tr_0):\bY\to\bZ_1\times\bZ_2$ is an isomorphism too.
Since 
\[
\Big[h\mapsto \Big(-3\lambda^2\gamma H'^2h_p, -(1-\p_q^2)^{-1}\Big[\tr_0h-\frac{\lambda^4}{\alpha}\Big(S[h]-\int_0^1S[h]\, {\rm d} q\Big)\Big]\Big)\Big]:\bY\to\bZ_1\times\bZ_2
\]
is a compact operator, the desired claim for $\partial_h\cF(\lambda ,0) $ follows at once.
\end{proof}

The next lemma characterizes the functions that belong to the kernel of $\partial_h\cF(\lambda,0).$

\begin{lemma}\label{L:ib}
Assume that \eqref{COND1Thm} is satisfied and set
\begin{equation}\label{eqa}
a:=1/H',
\end{equation}
where $H$ is the unique solution to \eqref{LFS}.
  Then, given  $\lambda>0$,  $h\in\bY$ satisfies $\partial_h\cF(\lambda,0)[h] = 0$ iff $h_0=0$ and for all $1\leq k\in\N$ we have
\begin{equation*}
\left.
\arraycolsep=1.4pt
\begin{array}{rllllll}
	\lambda^2 (a^3h_k')'  -(2k\pi)^2 a h_k &=& 0 \quad \text{in ${\rm C}^\beta([p_0,0])$},\\[1ex]
	\big( g\lambda^4    +\alpha(2k\pi)^4 \big)h_k(0) &=& \lambda^4a^3(0)h_k'(0),\\[1ex]
	h_k(p_0)&=&0,
\end{array}
\right\}
\end{equation*}
where the function $h_k\in {\rm C}^{2+\beta}([p_0,0])$, $k\in\N$,  is defined by
\begin{align}\label{hk}
	h_k(p) := \int_0^1 h(q,p) \cos(2k\pi q) \,dq,\qquad p\in[p_0,0].
\end{align}
\end{lemma}
\begin{proof}
Let  $h\in\bY$ satisfy $\partial_h\cF(\lambda,0)[h] = 0$. 
Then, the relation $L [h]=0$ is equivalent to
\begin{align*}
	\lambda^2 h_k''-3\lambda^2\gamma H'^2h_k' -(2k\pi)^2H'^2h_k = 0 \qquad \text{in ${\rm C}^\beta([p_0,0])$ for all $k\in\N$}.
\end{align*}
The latter identity is obtained by multiplying the equation $L[h]=0$ by $\cos(2k\pi q)$, followed by integration on~${[p_0,0]}.$
Since  $a$ is positive  and $\gamma=-aa',$ cf.~\eqref{LFS}$_1$, we may reformulate the latter equation as
\begin{align*}
	\lambda^2 (a^3h_k')'  -(2k\pi)^2 a h_k = 0 \qquad \text{in ${\rm C}^\beta([p_0,0])$ for all $k\in\N$}.
\end{align*}
Furthermore, since $h\in\bY$ we have
\begin{align*}
	h_0(0) = 0,
\end{align*}
while, arguing similarly as above,  the relation $T[h]=0$ is equivalent to
\begin{align*}
	\big( g\lambda^4    + \alpha(2k\pi)^4 \big)h_k(0) = \lambda^4 a^3(0) h_k'(0)\qquad\text{for all $k\ge 1$}.
\end{align*}

Finally, since each $h\in\bY$ vanishes on the boundary $p=p_0,$ it holds that 
\[
h_k(p_0)=0\qquad\text{for all $k\in\N$.}
\]

Noticing that the function $h_0\in {\rm C}^{2+\beta}([p_0,0])$ solves  the boundary value problem
\[
(a^3h_0')' = 0 \quad \text{in $[p_0,0]$},\qquad h_0(0)=h_0(p_0)=0,
\]
it is straightforward to conclude  that actually $h_0=0$.
This proves the claim.
\end{proof}

In  Lemma~\ref{L:ib} we have shown that a function  $h\in\bY$ with $h_0=0$ solves~${\partial_h\cF(\lambda,0)[h] = 0}$ iff for all~${k\geq1}$ the 
function~$h_k$ defined in \eqref{hk} is a solution to the Sturm-Liouville problem
\begin{equation}\label{BVPk}
\left.
\arraycolsep=1.4pt
\begin{array}{rllllll}
\lambda^2 (a^3 f')' -\mu a f &=&0\quad\text{in $[p_0,0]$,}\\[1ex]
( g\lambda^4 + \alpha\mu^2 )f(0) &=&\lambda^4 a^3(0) f'(0), \\[1ex]
f(p_0)&=&0 
\end{array}
\right\}
\end{equation}
with $\mu:=(2k\pi)^2$.
We next determine $\lambda_*>0$ such that   \eqref{BVPk} has a nontrivial solution for~${\mu=(2\pi)^2}$  and only the trivial solution $f=0$ when $\mu>(2\pi)^2$.
As a first step we show that the solutions to \eqref{BVPk} build a vector space of dimension less or equal to $1$ for each choice of the parameters $\lambda>0$ and~${\mu\in\R}$. 
To this end we define the   Sturm--Liouville type operator $R_{\lambda,\mu}:{\rm C}^{2+\beta}_0([p_0,0])\to {\rm C}^{\beta}([p_0,0])\times\R$ by 
\begin{align}\label{SLop}
R_{\lambda,\mu}[f]:=\begin{pmatrix}
\lambda^2 (a^3 f')' -\mu a f \\[1ex]
  \lambda^4 a^3(0) f'(0)-( g\lambda^4    + \alpha\mu^2 )f(0)
\end{pmatrix},
\end{align}
where ${\rm C}^{2+\beta}_0([p_0,0]):=\{f\in{\rm C}^{2+\beta} ([p_0,0])\,:\, f(p_0)=0\}.$
Let further~${f_1,\, f_2\in{\rm C}^{2+\beta}_0([p_0,0])}$ denote the   solutions to the initial value problems
\begin{equation}\label{SYS1}
\left.
\begin{array}{lll}
\lambda^2 (a^3 f_1')' -\mu a f_1 = 0\qquad \text{in $[p_0,0]$},\\[1ex]
 f_1(p_0)=0,\quad  f_1'(p_0)=1,
\end{array}
\right\}
\end{equation}
and 
\begin{equation}\label{SYS2}
\left.
\begin{array}{lll}
\lambda^2 (a^3 f_2')' -\mu a f_2 = 0\qquad \text{in $[p_0,0]$},\\[1ex]
f_2(0)=   \lambda^4 a^3(0),\quad f_2'(0)=  g\lambda^4    +\alpha\mu^2.
\end{array}
\right\}
\end{equation}

\begin{lemma}\label{L:3}
Given $\lambda>0$  and $\mu\in\R$, the operator~$R_{\lambda,\mu}$  
is a Fredholm operator of index zero with ${\rm dim} \ker R_{\lambda,\mu}\le 1$.
Additionally,   ${\rm dim} \ker R_{\lambda,\mu}= 1$ iff  the functions~$f_1$ and $f_2$    
are linearly dependent. In this case  we have $\ker R_{\lambda,\mu}={\rm span}\{f_1\}.$
\end{lemma}
\begin{proof}
Since $[f\mapsto (\lambda^2 (a^3 f')',  \lambda^4 a^3(0) f'(0))]:{\rm C}^{2+\beta}_0([p_0,0])\to {\rm C}^{\beta}([p_0,0])\times\R$    is obviously an isomorphism
and~$R_{\lambda,\mu}$ is a compact perturbation of this operator, it follows that
  $R_{\lambda,\mu} $ is indeed a Fredholm operator of index zero.
Moreover, if $f,\,\wt f\in \ker R_{\lambda,\mu}$, then $ f'\wt f=f\wt f'$, which shows that ${\rm dim} \ker R_{\lambda,\mu}\leq 1$.

Let  ${\rm dim} \ker R_{\lambda,\mu}= 1$ and let $0\neq f\in  \ker R_{\lambda,\mu}$.  
Then, since~$a^3(ff_i'-f'f_i)$, $i=1,\, 2,$  is a  constant function, the initial conditions in  \eqref{SYS1}-\eqref{SYS2} 
ensure that this function is in fact identically zero, hence $f_1$ and $f_2$ are linearly dependent.
Viceversa, if $f_1$ and $f_2$ are linearly dependent, then they both belong to $ \ker R_{\lambda,\mu},$ and this completes the proof.
\end{proof}

In view of Lemma~\ref{L:3} it remains to look for a value  $\lambda_*>0$ with the property that the Wronskian~${f_1f_2'-f_1'f_2}$ vanishes in $[p_0,0]$ only for $\mu=(2\pi)^2$.  
Since $a^3(f_1f_2'-f_1'f_2)$ is a constant function in $[p_0,0]$ and $a\neq 0,$ the Wronskian 
vanishes in $[p_0,0]$ iff it vanishes at the point $p=0$.
Therefore we consider the function $W:(0,\infty)\times\R\to\R$ given by
\[
W(\lambda,\mu)=f_1(0)f_2'(0)-f_1'(0)f_2(0)=(g\lambda^4    + \alpha\mu^2)f_1(0)-\lambda^4 a^3(0)f_1'(0).
\]
Observing that   \eqref{SYS1}$_1$ depends smoothly on $\mu$ and $\lambda$, this property is inherited also by the solution to~\eqref{SYS1}, cf. e.g.~\cite{Am90}, and therefore we have $W\in{\rm C}^\infty((0,\infty)\times\R).$

For the special value $\mu=0$ we have
\[
W(\lambda,0)=\lambda^4(g   f_1(0)- a^3(0)f_1'(0)).
\]
In view of \eqref{SYS1}  we compute, in the particular case $\mu=0$, that 
\[
f_1'(0)=\frac{a^3(p_0)}{a^3(0)}\qquad\text{and}\qquad f_1(0)=\int_{p_0}^0\frac{a^3(p_0)}{a^3(p)}{\rm d}p,
\]
which leads to
\begin{equation}\label{wloo}
W(\lambda,0)=\lambda^4a^3(p_0)\Big(g   \int_{p_0}^0 \frac{1}{a^3(p)}{\rm d}p- 1\Big).
\end{equation}
Hence, if 
\begin{equation}\label{COND2}
g  \int_{p_0}^0\frac{1}{a^3(p)}{\rm d}p< 1,
\end{equation}
then  
\begin{equation}\label{Wl0}
W(\lambda,0)<0\qquad\text{for all $\lambda>0$.}
\end{equation}

We next investigate the behavior of $W(\lambda,\mu)$ when $\mu\to\infty$ .
\begin{lemma}\label{L:4}
Given $\lambda>0$, it holds that $\underset{\mu\to\infty}\lim W(\lambda,\mu)=\infty.$
\end{lemma}
\begin{proof}
Let  $m,\, M\in(0,\infty)$ be defined as $m:=\min_{[p_0,0]} a$ and $M:=\max_{[p_0,0]}a$.
We note that, since~${f_1'(p_0)=1}$ and~${f_1(p_0)=0}$,  there exists $\ov p\in(0,p_0]$ such that $f_1(p)>0$ in~${(p_0,\ov p).}$
This property together with  \eqref{SYS1}$_1$  implies that   $a^3f_1'$ is a non-decreasing function in $(p_0,\ov p)$ and that~${f_1'(p)\geq (m/M)^3}$ for all $p\in[p_0,\ov p]$. 
Consequently, by the fundamental theorem of calculus, $f_1(\ov p)\geq (\ov p-p_0) (m/M)^3>0$ and therefore we may actually chose~${\ov p=0}$. 
Thus,      $f_1$ is an  increasing function in $[p_0,0]$   and $f_1(0)\geq |p_0|(m/M)^3.$

Integrating  the first equation of \eqref{SYS1} over $[p_0,0]$, we have
\[
a^3(0)f_1'(0)=a^3(p_0)+\frac{\mu}{\lambda^2}\int_{p_0}^0a(p)f_1(p)\,{\rm d}p\leq M^3+\frac{M|p_0|\mu}{\lambda^2} f_1(0)\qquad\text{for all  $\lambda,\, \mu\in (0,\infty)$.} 
\] 
This estimate together with  the definition of $W(\lambda,\mu)$ and the relation $f_1(0)\geq |p_0|(m/M)^3$     implies that 
\[
W(\lambda,\mu)\geq   (\alpha\mu^2- \lambda^2 M|p_0|\mu  ) f_1(0)-\lambda^4 M^3\underset{\mu\to\infty}\longrightarrow\infty,
\]
and the desired claim follows. 
\end{proof}

From now on, we assume that \eqref{COND2} is satisfied.
In view of Lemma~\ref{L:4}, for each given~${\lambda>0}$, the function $W(\lambda,\cdot)$ has at least a positive zero. 
We next investigate the partial derivatives $W_\lambda(\lambda,\mu)$ and $W_\mu(\lambda,\mu)$ in all points  $(\lambda,\mu)$ having the property that~${W(\lambda,\mu)=0.}$

\begin{lemma}\label{L:5} 
Let   $\lambda,\, \mu\in (0,\infty)$ be given such that $W(\lambda,\mu)=0$. 
We then have:
\begin{itemize}
\item[(i)]  $W_\lambda(\lambda,\mu)<0$;
\item[(ii)] $W_\mu(\lambda,\mu)>0$ and 
\begin{equation}\label{wonder}
\frac{W_\lambda(\lambda,\mu)}{W_\mu(\lambda,\mu)}=-\frac{2\mu}{\lambda}.
\end{equation}
\end{itemize}
\end{lemma}
\begin{proof}
If $W(\lambda,\mu)=0$, the solutions $f_1$ and $f_2$  to \eqref{SYS1} and \eqref{SYS2} are linearly dependent, hence  $f_1=\Theta f_2$ with $  \Theta\in\R.$
By Lemma~\ref{L:4},   $f_1 $   is positive in $(p_0,0]$. 
Recalling that also~${f_2(0)>0}$, it follows that actually $\Theta>0$. 

Given $(\lambda,\mu)\times\R\in(0,\infty)$, an application of the chain rule yields that
\[
W_\lambda(\lambda,\mu)=(g\lambda^4    + \alpha\mu^2)f_{1,\lambda}(0)-\lambda^4 a^3(0)f_{1,\lambda}'(0) +4\lambda^3g   f_1(0)-4\lambda^3a^3(0)f_1'(0),
\] 
where $f_{1,\lambda}\in{\rm C}^{2+\beta}([p_0,0])$ is the solution to
\begin{equation}\label{SYS1l}
\left.
\begin{array}{lll}
\lambda^2 (a^3 f_{1,\lambda}')' -\mu a f_{1,\lambda} = -2\lambda (a^3 f_1')'\qquad \text{in $[p_0,0]$},\\[1ex]
 f_{1,\lambda}(p_0)=0,\quad  f_{1,\lambda}'(p_0)=0.
\end{array}
\right\}
\end{equation}
We next multiply \eqref{SYS1l}$_1$ by $f_1$ and subtract from this relation the  identity \eqref{SYS1} multiplied with $f_{1,\lambda},$ to obtain, after integration on $[p_0,0],$ that 
\[
\lambda a^3(0)f_1(0)f_{1,\lambda}'(0)=\lambda a^3(0)f_1'(0)f_{1,\lambda}(0)-2 a^3(0)f_1(0)f_{1}'(0)+2\int_{p_0}^0a^3(p)f_1'^2(p)\,{\rm d}p.
\]
If $W(\lambda,\mu)=0$, the latter relation together with the identity $f_1=\Theta f_2$, $\Theta>0$, leads us to 
\begin{equation*}
\begin{aligned}
f_1(0)W_\lambda(\lambda,\mu)&=2\lambda^3\Big(2gf_1^2(0)-a^3(0)f_1(0)f_1'(0)-\int_{p_0}^0a^3(p)f_1'^2(p)\,{\rm d}p\Big)\\[1ex]
&=2\lambda^3\Big( a^3(0)f_1(0)f_1'(0)-\frac{2\alpha\mu^2}{\lambda^4}f_1(0)^2-\int_{p_0}^0a^3(p)f_1'^2(p)\,{\rm d}p\Big).
\end{aligned}
\end{equation*}
We may now  multiply~\eqref{SYS1}$_1$ by~$f_1$ and integrate over $[p_0,0] $ to obtain that 
\begin{equation}\label{refer}
a^3(0)f_1(0) f_1'(0)=\frac{\mu}{\lambda^2}\int_{p_0}^0  a(p)f_1^2(p)\,{\rm d}p+\int_{p_0}^0 a^3(p)f_1'^2(p) \,{\rm d}p,
\end{equation}
hence, on the one hand
\begin{equation}\label{aform}
f_1(0)W_\lambda(\lambda,\mu)=2\lambda^3\Big( \frac{\mu}{\lambda^2}\int_{p_0}^0  a(p)f_1^2(p)\,{\rm d}p-\frac{2\alpha\mu^2}{\lambda^4}f_1(0)^2\Big),
\end{equation}
and, on the other hand
\begin{equation}\label{aform2}
f_1(0)W_\lambda(\lambda,\mu) =2\lambda^3\Big(2gf_1^2(0)- 2\int_{p_0}^0 a^3(p)f_1'^2(p) \,{\rm d}p-\frac{\mu}{\lambda^2}\int_{p_0}^0  a(p)f_1^2(p)\,{\rm d}p\Big).
\end{equation}
H\"older's inequality  now yields 
 \begin{equation}\label{baba}
 gf_1^2(0)=g\Big(\int_{p_0}^0f_1'(p)\,{\rm d}p\Big)^2\leq g\Big(\int_{p_0}^0\frac{1}{a^3(p)}\,{\rm d}p\Big)\Big(\int_{p_0}^0a^3(p)f_1'^2(p)\,{\rm d}p\Big)
 \end{equation}
 and together with \eqref{aform2} and \eqref{COND2} we have
 \begin{equation*} 
\frac{f_1(0)W_\lambda(\lambda,\mu)}{ 4\lambda^3}<  \Big(g\int_{p_0}^0\frac{1}{a^3(p)}\,{\rm d}p-1\Big) \Big(\int_{p_0}^0a^3(p)f_1'^2(p)\,{\rm d}p\Big)<0,
\end{equation*}
which proves (i).

In order to prove (ii), we proceed similarly as above and compute, for  given~${\lambda,\,\mu\in(0,\infty)}$, that
\[
W_\mu(\lambda,\mu)=(2g\lambda^4    - \alpha\mu^2)f_{1,\mu}(0)-\lambda^4 a^3(0)f_{1,\mu}'(0) -2\alpha\mu f_1(0),
\] 
where $f_{1,\mu}\in{\rm C}^{2+\beta}([p_0,0])$ is the solution to
\begin{equation}\label{SYS1m}
\left.
\begin{array}{lll}
\lambda^2 (a^3 f_{1,\mu}')' -\mu a f_{1,\mu} = af_1\qquad \text{in $[p_0,0]$},\\[1ex]
 f_{1,\mu}(p_0)=0,\quad  f_{1,\mu}'(p_0)=0.
\end{array}
\right\}
\end{equation}
We next multiply \eqref{SYS1m}$_1$ by $f_1$ and subtract from this relation the  identity \eqref{SYS2} multiplied with $f_{1,\mu},$ to obtain, after integration on $[p_0,0],$ that 
\[
  a^3(0)f_1(0)f_{1,\mu}'(0)= a^3(0)f_1'(0)f_{1,\mu}(0)+\frac{1}{\lambda^2}\int_{p_0}^0a(p)f_1^2(p)\,{\rm d}p.
\]
Hence, if $W(\lambda,\mu)=0$, the latter identity combined with the relation $f_1=\Theta f_2$, $\Theta>0$, leads us to 
\begin{equation*}
\frac{\mu f_1(0)W_\mu(\lambda,\mu)}{\lambda^4}=-\Big( \frac{\mu}{\lambda^2}\int_{p_0}^0  a(p)f_1^2(p)\,{\rm d}p-\frac{2\alpha\mu^2}{\lambda^4}f_1(0)^2\Big),
\end{equation*}
and (ii) follows in view of \eqref{aform} and (i).
\end{proof}

Given $\lambda>0$, let 
\begin{equation}\label{defmu}
\mu(\lambda):=\inf\{\mu>0\,:\, W(\lambda,\mu)>0\}.
\end{equation}
Since $W$ is smooth, \eqref{Wl0} and Lemma~\ref{L:4} imply that $\mu(\lambda)$ is well-defined for all $\lambda>0$ and moreover~$\mu(\lambda)>0$.
In fact, Lemma~\ref{L:5} implies that, for each $\lambda>0,$ $\mu(\lambda)$ is the unique positive zero of the mapping $W(\lambda,\cdot)$. 
Moreover, the implicit function theorem together with Lemma~\ref{L:5} ensures that 
\[
[\mu\mapsto\lambda(\mu)]:(0,\infty)\to(0,\infty) 
\]
is smooth.
We now use the chain rule together with \eqref{wonder} to compute that
\[
\mu'(\lambda)=-\frac{W_\lambda(\lambda,\mu)}{W_\mu(\lambda,\mu)}=\frac{2\mu(\lambda)}{\lambda},\qquad\lambda>0,
\]
from where we infer that there exists a positive constant $C_0$ such that 
\begin{equation}\label{c0}
\mu(\lambda)=C_0\lambda^2,\qquad\lambda>0.
\end{equation}
It is remarkable that exactly the same expression for $\mu$ (with a possibly different constant~$C_0$) has been obtained in \cite{EKLM20} in the analysis of the bifurcation problem for stratified capillary-gravity waves.
We arrive at the following result.

 \begin{lemma}\label{L:6} Let 
 \begin{align}\label{lambda*}
  \lambda_*:=\frac{2\pi}{\sqrt{C_0}},
 \end{align}
 where $C_0$ is the constant identified in \eqref{c0}.
 Then, $\p_h\cF(\lambda_*,0)$ is a Fredholm operator of index zero and with a one-dimensional kernel spanned by
 \begin{align}\label{h*}
 h_*(q,p):=f_{1,*}(p)\cos(2\pi q), \qquad (q,p)\in\ov\0, 
 \end{align}
 where $f_{1,*} $ denotes the solution to \eqref{SYS1} corresponding to the parameters $(\lambda,\mu)=(\lambda_*,(2\pi)^2).$
 \end{lemma}
 \begin{proof}
 Since   $\mu(\lambda_*)=(2\pi)^2$ is the unique positive zero of $W(\lambda_*,\cdot),$ see  \eqref{defmu}-\eqref{lambda*},  the claim follows from Lemma~\ref{L:2}, Lemma~\ref{L:ib}, and Lemma~\ref{L:3}.
 \end{proof}

 In order to apply  \cite[Theorem~1.7]{CR71} in the context of the  bifurcation problem \eqref{BifE}, it remains to prove that  the transversality condition
 \begin{equation}\label{TC}
 \p_{\lambda h}\cF(\lambda_*,0)[h_*]\not\in \ima \p_{h}\cF(\lambda_*,0),
 \end{equation}
  with $\lambda_*$ and $h_*$ introduced in \eqref{lambda*} and \eqref{h*}, is satisfied.
Therefore, we first characterize in Lemma~\ref{L:7} below the range of $  \p_{h}\cF(\lambda_*,0).$

 \begin{lemma}\label{L:7}
 A pair $(F,\varphi)\in \bZ_1\times\bZ_2$  belongs to $\ima\p_{h}\cF(\lambda_*,0)$ iff
 \begin{align}\label{IF0}
\int_\0 a^3h_*F\,{\rm d}(q,p)+\alpha C_0(1+C_0\lambda_*^2)\int_0^1 \varphi\tr_0h_*\,{\rm d}q=0,
 \end{align}
 where $h_*\in \ker\p_{h}\cF(\lambda_*,0)$ is defined in \eqref{h*}.
  \end{lemma}
  \begin{proof}
  As a starting point we observe  that  for $\p_{h}\cF(\lambda_*,0)=(L,T)$ we have
  \[
L[h]=\lambda_* ^2 a^{-3}(a^3h_p)_p+a^{-2}h_{qq},\qquad h\in\bY.  
  \]
 We now assume that $(F,\varphi)\in \ima\p_{h}\cF(\lambda_*,0)$ is the image of a function $h\in\bY$.
 Multiplying the identity $L[h]=F$ by $a^3h_*$ and integrating the resulting relation by parts, we find that
 \begin{equation}\label{78a}
 \int_\0 a^3h_*F\,{\rm d}(q,p)=\int_0^1\lambda_*^2\tr_0(a^3h_ph_*)\,{\rm d}q-\int_\0 \big[\lambda_*^2 a^3 h_p h_{*,p}+ah_qh_{*,q}\big]\,{\rm d}(q,p).
 \end{equation}
 Moreover, after multiplying the relation $T[h]=\varphi$ by $\tr_0h_*$ and integrating the resulting relation by parts, we obtain,  in view of the symmetry of the operator $(1-\p_q^2)^{-1},$ that 
  \begin{equation*}
  \int_0^1\lambda_*^4\tr_0(a^3h_ph_*)\,{\rm d}q= \int_0^1(\alpha (2\pi)^4+g\lambda_*^4)\tr_0( hh_*)\,{\rm d}q-\alpha(2\pi)^2(1+(2\pi)^2)\int_0^1 \varphi\tr_0h_*\,{\rm d}q.
 \end{equation*}
In virtue of \eqref{c0}, \eqref{lambda*}, and \eqref{SYS1} we have 
\begin{equation}\label{desis}
(\alpha (2\pi)^4+g\lambda_*^4)\tr_0 h_*=\lambda^4_*\tr_0(a^3h_{*,p}) 
\end{equation}
and the latter identity can thus be recast as
\begin{equation}\label{78b}
  \int_0^1\lambda_*^2\tr_0(a^3h_ph_*)\,{\rm d}q=\int_0^1\lambda^2_*\tr_0(a^3h_{*,p}h)\,{\rm d}q-\alpha C_0(1+(2\pi)^2)\int_0^1 \varphi\tr_0h_*\,{\rm d}q,
 \end{equation}
 where $C_0$ is defined in \eqref{c0}.
 We now sum up \eqref{78a} and \eqref{78b} to conclude that 
 \begin{equation}\label{78c}
 \begin{aligned}
& \int_\0 a^3h_*F\,{\rm d}(q,p)+\alpha C_0(1+(2\pi)^2)\int_0^1 \varphi\tr_0h_*\,{\rm d}q\\[1ex]
 &\qquad=\int_0^1\lambda^2_*\tr_0(a^3h_{*,p}h)\,{\rm d}q-\int_\0 \big[\lambda_*^2 a^3 h_p h_{*,p}+ah_qh_{*,q}\big]\,{\rm d}(q,p).
 \end{aligned}
 \end{equation}
 Moreover, arguing as above, but interchanging the roles of $h$ and $h_*$, we arrive, in view of~${\p_{h}\cF(\lambda_*,0)[h_*]=0,}$ at
  \begin{equation*}
 \begin{aligned}
 &\int_0^1\lambda^2_*\tr_0(a^3h_{*,p}h)\,{\rm d}q-\int_\0 \big[\lambda_*^2 a^3 h_p h_{*,p}+ah_qh_{*,q}\big]\,{\rm d}(q,p)=0.
 \end{aligned}
 \end{equation*}
This relation together with \eqref{78c} immediately  implies \eqref{IF0}. 

Noticing that \eqref{IF0} defines a closed subspace of $\bZ_1\times\bZ_2$ of codimension 1 which contains the range of~${\p_{h}\cF(\lambda_*,0),}$ the desired claim follows now from Lemma~\ref{L:2}. 
  \end{proof}

We are now in a position to verify the transversality condition \eqref{TC}.

\begin{lemma}\label{L:8}
We have $\p_{\lambda h}\cF(\lambda_*,0)[h_*]\not\in \ima \p_{h}\cF(\lambda_*,0)$.
  \end{lemma}
  \begin{proof} 
  Recalling \eqref{derf}, we have
   \begin{equation*} 
 \begin{aligned}
 \p_{\lambda h}\cF_1(\lambda_*,0)[h_*]&:=2\lambda_* a^{-3}(a^3h_{*,p})_p,\\
  \p_{\lambda h}\cF_2(\lambda_*,0)[h_*]&:= \frac{4\lambda_*^3}{\alpha}(1-\p_q^2)^{-1}\Big[ S [h_*]-\int_0^1S[h_*]\, {\rm d} q\Big],
 \end{aligned}
 \end{equation*}
  where, using the definition of the operator $S$ and that of $h_*$ together with \eqref{desis}, we have
\[
S [h_*]-\int_0^1S[h_*]\, {\rm d} q=-\frac{1}{(2\pi)^2}\tr_0(a^3h_{*,p}-gh_*)=-\frac{\alpha C_0^2}{(2\pi)^2}\tr_0 h_*.
\]
Appealing to \eqref{refer}, we compute
\begin{align*}
\int_\0 a^3h_* \p_{\lambda h}\cF_1(\lambda_*,0)[h_*]\,{\rm d}(q,p)&=\lambda_*\Big(a^3(0)f_{1,*}(0)f_{1,*}'(0)-\int_{p_0}^0 a^3(p)f_{1,*}'^2(p)\,{\rm d}p\Big)\\[1ex]
&=\lambda_*\Big((g+\alpha C_0^2)f_{1,*}^2(0) -\int_{p_0}^0 a^3(p)f_{1,*}'^2(p)\,{\rm d}p\Big)
\end{align*}
and
\begin{align*}
\alpha C_0(1+C_0\lambda_*^2)\int_0^1 \p_{\lambda h}\cF_2(\lambda_*,0)[h_*]\tr_0h_*\,{\rm d}q&=2\lambda_*f_{1,*}(0)(a^3(0)f_{1,*}'(0)-2gf_{1,*}(0))\\[1ex]
&=-2\alpha\lambda_* C_0^2f_{1,*}^2(0),
 \end{align*}
hence, using also \eqref{baba} and \eqref{COND2}, we get
 \begin{align*}
&\int_\0 a^3h_* \p_{\lambda h}\cF_1(\lambda_*,0)[h_*]\,{\rm d}(q,p)+\alpha C_0(1+C_0\lambda_*^2)\int_0^1 \p_{\lambda h}\cF_2(\lambda_*,0)[h_*]\tr_0h_*\,{\rm d}q\\[1ex]
&\qquad=\lambda_*\Big((g-\alpha C_0^2)f_{1,*}^2(0) -\int_{p_0}^0 a^3(p)f_{1,*}'^2(p)\,{\rm d}p\Big)\\[1ex]
&\qquad<\lambda_*\Big(gf_{1,*}^2(0) -\int_{p_0}^0 a^3(p)f_{1,*}'^2(p)\,{\rm d}p\Big)\\[1ex]
&\qquad\leq\lambda_*\Big( g \int_{p_0}^0\frac{1}{a^3(p)}\,{\rm d}p-1\Big)\int_{p_0}^0 a^3(p)f_{1,*}'^2(p)\,{\rm d}p<0.
 \end{align*}
This proves the claim.
  \end{proof}

We are now in a position to establish Theorem~\ref{MT1}.

\begin{proof}[Proof of Theorem~\ref{MT1}]
In view of Lemma~\ref{L:1} and of Proposition~\ref{P:Ep}, in the framework of waves which are symmetric with respect to the vertical line $x=0$, 
the Euler formulation~\eqref{For1} of the steady   hydroelastic waves problem 
is equivalent to the height function formulation~\eqref{PBF}, hence also to the bifurcation problem \eqref{BifE}.
Therefore, the assertion (i) is a straightforward consequence of Lemma~\ref{L:0}.

Concerning (iia), if \eqref{COND2Thm} is not satisfied, then $W(\lambda,0)\geq0$ for all $\lambda>0,$ see \eqref{wloo}, and Lemma~\ref{L:4} and Lemma~\ref{L:5}~(ii) then ensure that $W(\lambda,\mu)>0$ for all $\mu>0$.
Lemma~\ref{L:2}, Lemma~\ref{L:ib}, and Lemma~\ref{L:3} then imply  that $\p_h\cF(\lambda,0) $ is an isomorphism, hence $(\lambda,0)$ is not a bifurcation point for \eqref{BifE}, regardless of the value of $\lambda>0$.

It remains to establish (iib).
Therefore we note that \eqref{wloo}, \eqref{COND2}, Lemma~\ref{L:4}, and Lemma~\ref{L:5}~(ii) imply that for each $\lambda>0$, the nonlinear equation $W(\lambda,\cdot)=0$ has a 
unique solution $\mu=\mu(\lambda)>0$, which is given by \eqref{c0}.
Our previous requirements  that the boundary value problem~\eqref{BVPk} has a nontrivial solution for~$\mu=(2\pi)^2$ and only the zero solution for $\mu>(2 \pi)^2$ identifies a unique value $\lambda_*$, see~\eqref{lambda*}, with this property.
The smoothness property \eqref{refF} together with Lemma~\ref{L:6} and Lemma~\ref{L:8} enable us now to use the local bifurcation theorem due to Crandall and Rabinowitz, 
cf.  \cite[Theorem~1.7]{CR71}, in the context of \eqref{BifE},  to conclude, in view of the equivalence of the formulations \eqref{For1} and~\eqref{BifE}, 
the existence of the smooth local  bifurcation curve 
   \begin{align}\label{BCu1}
      [s\mapsto (\lambda(s),h(s))]:(-\e,\e)\to (0,\infty)\times \cO,
   \end{align}
where  $\e>0$  is  small,  such that $\mathcal{F}(\lambda(s),h(s))=0$ for all $|s|<\e$.
 Moreover, $\lambda(0)=\lambda_*$ and 
  \begin{align}\label{AAS0}
  h(s)=s(h_*+\chi(s)) 
  \end{align}
  where $\chi\in{\rm  C}^\infty((-\e,\e),\bY)$ satisfies $\chi(0)=0$.
Arguing similarly  as in \cite[Section 5]{CoSt04}, it is not difficult to prove that, since the function $h_*$ satisfies~${h_*(q,0)=f_{1,*}(0)\cos(2\pi q),\, q\in\R}$, see~\eqref{h*}, with $f_{1,*}(0)>0$, also the 
 waves profile  $ \eta(s)$ has for~${s\neq0}$ exactly one  maximum (at $x=0$) and  minimum (at $x=\lambda(s)/2$) per period. This completes the proof.
\end{proof}

\bibliographystyle{siam}
\bibliography{Literature}
\end{document}